\documentclass[11pt]{amsart}
\usepackage{amsmath,amssymb,amsthm,amsfonts,color,longtable}
\usepackage{hyperref}
\usepackage[margin=3cm]{geometry}

\numberwithin{equation}{section}

\theoremstyle{plain}

\newtheorem*{acknowledgements}{Acknowledgements}

\numberwithin{equation}{section}
\newtheorem{thm}{Theorem}[section]
\newtheorem{prop}[thm]{Proposition}
\newtheorem{lemma}[thm]{Lemma}
\newtheorem{cor}[thm]{Corollary}

\theoremstyle{remark}
\newtheorem{remark}[thm]{Remark}

\theoremstyle{definition}
\newtheorem*{notation}{Notation}

\providecommand{\abs}[1]{\left\lvert#1\right\rvert}

\providecommand{\ip}[1]{\langle#1\rangle}
\providecommand{\iFT}[1]{\mathcal{F}^{-1}\left(#1\right)}

\renewcommand{\div}{{\rm div}}
\newcommand{\curl}{{\rm curl}} 
  
\newcommand{\C}{\mathbb{C}}  
\newcommand{\N}{\mathbb{N}} 
\newcommand{\R}{\mathbb{R}}
\newcommand{\PP}{\mathbb{P}}

\title{Convergence to Stratified Flow for an Inviscid 3D Boussinesq System}
\author{Klaus Widmayer}

\date{\today}

\address{Courant Institute of Mathematical Sciences, 251 Mercer Street, New York 10012 NY, USA}
\email{klaus@cims.nyu.edu}

\begin{document}

\begin{abstract}
We study the stability of special, stratified solutions of a 3d Boussinesq system describing an incompressible, inviscid 3d fluid with variable density (or temperature, depending on the context) under the effect of a uni-directional gravitational force. The behavior is shown to depend on the properties of an anisotropic dispersive operator with weak decay in time. However, the dispersive decay also depends on the strength of the gravity in the system and on the profile of the stratified solution, whose stability we study. We show that as the strength of the dispersion in the system tends to infinity, the 3d system of equations tends to a stratified system of 2d Euler equations with stratified density.
\end{abstract}

\maketitle

\tableofcontents

\section{Introduction}
We study an incompressible, inviscid fluid $u:\R^+\!\!\times\R^3\rightarrow\R^3$ of variable density or temperature (depending on the physical context) $\theta:\R^+\!\!\times\R^3\rightarrow\R$ under the influence of an external gravity force proportional to $\theta$ and acting in the third coordinate direction $\vec{e}_3=\left(\begin{array}{c} 0\\0\\1\end{array}\right)$ of $\R^3$, described by Euler's equation coupled to a continuity equation and an equation of state (see e.g.\ \cite{MR1965452}, \cite{MR1925398}):
\begin{equation}\label{eq:BQ0}
\begin{cases}
 &\partial_t u+u\cdot\nabla u+\nabla p=\nu \Delta u+\kappa^2 \theta\vec{e}_3
,\\
 &\partial_t \theta + u\cdot\nabla \theta=\mu\Delta \theta,\\
 &\div \,u=0.
\end{cases}
\end{equation}
Here $p:\R^+\!\!\times\R^3\rightarrow\R$ is the fluid pressure, $\kappa>0$ is a gravitational constant and $\nu,\mu>0$ are viscosity parameters.

There is too vast a number of results regarding various aspects of this problem to be surveyed here. We just point out that the question of global well-posedness of this system remains a major open problem, so most of the work so far has focused on various improvements over the standard local well-posedness theory (such as regularity criteria in various settings \cite{RegCritBQ}, \cite{MR2514915}, or blow-up conditions \cite{MR2653751}, to name but a few) or reductions to a two-dimensional problem. For the latter, a much more comprehensive understanding is available, including the construction of global solutions when $\nu\neq$ or $\mu\neq0$ (\cite{MR2227730}, \cite{MR2121245},\cite{MR2290277}, \cite{MR2782720}).

However, in the inviscid case ($\mu=\nu=0$), which will be the focus of this note, the question of global regularity is open, even for the two-dimensional problem, and only local results regarding stability, well-posedness in various function spaces and blow-up criteria (e.g.\ \cite{SQGpaper}, \cite{MR2645152}, \cite{MR3351981}) are known.

In the present article we investigate the behavior near a special stratified solution: We perturb \eqref{eq:BQ0} around $u=0$, $\theta=\lambda^2 z$ ($\lambda>0$) and obtain thus the system
\begin{equation}\label{eq:BQ_gen}
\begin{cases}
 &\partial_t u+u\cdot\nabla u+\nabla p=\nu \Delta u+\kappa^2 \theta\vec{e}_3,\\
 &\partial_t \theta + u\cdot\nabla \theta=\mu\Delta \theta-\lambda^2 u_3,\\
 &\div \,u=0.
\end{cases}
\end{equation}

\subsection{Inviscid Flow and Rescaling}
In order to study the dispersive effects in this system we focus now on the case of vanishing viscosity, i.e.\ $\nu=\mu=0$. It is natural\footnote{see Section \ref{sec:en_est}} for the energy estimates to rescale $\theta\mapsto T:=\frac{\kappa}{\lambda}\theta$ in \eqref{eq:BQ_gen}, which then becomes
\begin{equation}\label{eq:BQ_1}
\begin{cases}
 &\partial_t u+u\cdot\nabla u+\nabla p=\sigma T\vec{e}_3,\\
 &\partial_t T + u\cdot\nabla T=-\sigma u_3,\\
 &\div \,u=0,
\end{cases}
\end{equation}
with the new parameter $\sigma:=\kappa\lambda>0$.

As will be shown later, this system incorporates dispersive effects, the strength of which depends on the ``dispersion parameter'' $\sigma$.

\subsection{Main Theorem}
We are interested here in the question of the dynamics of \eqref{eq:BQ_1} as the dispersion gets increasingly strong. More precisely, we study the limit $\sigma\to\infty$. We will prove:
\begin{thm}\label{thm:main}
Consider a solution $(u^\sigma,T^\sigma)\in C([0,L], H^N)$ on a time interval $[0,L]$ to the initial value problem for \eqref{eq:BQ_1},
\begin{equation}\label{eq:BQ_2}
\begin{cases}
 &\partial_t u^\sigma+u^\sigma\cdot\nabla u^\sigma+\nabla p^\sigma=\sigma T^\sigma\vec{e}_3,\\
 &\partial_t T^\sigma + u^\sigma\cdot\nabla T^\sigma=-\sigma u^\sigma_3,\\
 &\div \,u^\sigma=0,\\
 &(u^\sigma,T^\sigma)(0)=(u_0,T_0),
\end{cases}
\end{equation}
satisfying 
\begin{equation}\label{eq:unif_energy}
\abs{(u^\sigma,T^\sigma)(t)}_{H^N}<\infty\text{ for }t\in [0,L] \text{ \emph{uniformly} in }\sigma,
\end{equation}
for some $N\geq 6$. Assume also that $\abs{(u_0,T_0)}_{W^{5,1}}<\infty$.

Then the solution $(u^\sigma,T^\sigma):\R^+\times\R^3\to\R^4$ can be decomposed into two pieces 
\begin{equation*}
 \left(\begin{array}{c}u^\sigma_1\\ u^\sigma_2\\ u^\sigma_3\\ T^\sigma\end{array}\right)=\left(\begin{array}{c}v^\sigma_1\\ v^\sigma_2\\ 0\\ 0\end{array}\right)+\left(\begin{array}{c}w^\sigma_1\\ w^\sigma_2\\ w^\sigma_3\\ T^\sigma\end{array}\right)
\end{equation*}
with the property that as $\sigma\to\infty$, for any $t\in (0,L]$ we have the convergences
$$(w^\sigma_1,w^\sigma_2,w^\sigma_3,T^\sigma)(t)\to 0 \text{ in } W^{1,\infty}(\R^3)$$
and
$$(v^\sigma_1,v^\sigma_2)(t)\to(\bar{u}_1,\bar{u}_2)(t) \text{ in } L^2(\R^3).$$ 
Here $\bar{u}:=(\bar{u}_1,\bar{u}_2):\R^+\times\R^3\to\R^2$ solves the system of two-dimensional\footnote{As described in Section \ref{sec:not}, the lower index $h$ denotes operation in only the ``horizontal'' variables $x_1,x_2$.} incompressible Euler equations
\begin{equation}\label{eq:2DEuler}
 \begin{cases}&\partial_t\bar{u}+\bar{u}\cdot\nabla_h\bar{u}+\nabla_h\bar{p}=0,\\ &\div_h(\bar{u})=0,\\ &\bar{u}(0)=(\bar{\PP}_0(u_0,T_0))_h,\end{cases}
\end{equation}
where $\bar{\PP}_0$ projects onto functions such that $\div_h(\bar{u})=0$.\footnote{For the explicit description of this projection see Lemma \ref{lem:eigen} and Section \ref{sec:reform}.}
\end{thm}

The rest of this note is dedicated to the proof of this result. Before we outline its main steps we give a few remarks on the context and relevance of this result.

\subsubsection*{Rate of Convergence}
 As the proof shows, the rate of convergence to the limiting system is not uniform on the full time interval $(0,L]$, but only away from the initial time (i.e.\ for any $\epsilon>0$ it is uniform on $[\epsilon,L]$). More precisely, for times $0\leq t\lesssim \sigma^{-1}$ there is a ``boundary layer'' where the bound for the dispersive decay is inconsequential -- see Remark \ref{rem:disp_param} on page \pageref{rem:disp_param}.
 
\subsubsection*{Initial Data of the Limiting System \eqref{eq:2DEuler}} 
 We note that the initial data for the limiting system of stratified 2d Euler equations \eqref{eq:2DEuler} are just the relevant horizontal components of those for the full system \eqref{eq:BQ_2}, the projection $\bar{\PP}_0$ ensuring that $\div_h(\bar{u}(0))=0$.
 
 In particular this means that in the limit there is no net effect from the dispersion on the horizontal motion of the fluid.
 
\subsubsection*{Comparison to Rotating Fluids}
 We point out that the limiting system \eqref{eq:2DEuler} is a \emph{stratified} system of 2d Euler equations, i.e.\ for any fixed $x_3\in\R$ the velocity $\bar{u}(t,x_1,x_2,x_3)$ solves a 2d Euler equation in the variables $t,x_1,x_2$. 
 
 This contrasts strongly with prior results on the Navier-Stokes-Coriolis system of a rotating fluid: In the work of Chemin et al.\ \cite{MR2228849} on a rotating 3d Navier-Stokes equation, the Coriolis force introduces dispersion into the system. However, in the limit of infinite dispersion (physically speaking, as the Rossby number tends to zero) one obtains a purely two-dimensional system: the velocity is independent of $x_3\in\R$. This is also known as \emph{columnar flow}.

\subsubsection*{Scaling}
 If $(u,T)$ solves \eqref{eq:BQ_1}, then so does $$(u_\lambda,T_\lambda)(t,x):=\lambda(u,T)(t,\lambda^{-1}x).$$ The invariant spaces for this scaling are thus $\dot{W}^{s,p}(\R^3)$ for $s=1+\frac{3}{p}$. In particular, the equation is critical in $\dot{H}^{\frac{5}{2}}$ and $\dot{W}^{1,\infty}$ and thus supercritical for the $L^2$ norm, which is conserved (see the energy equality \eqref{eq:en_eq}).
 
 For the convergences in Theorem \ref{thm:main} we point out that they take place in critical ($\dot{W}^{1,\infty}$) and supercritical ($L^\infty$, $L^2$) norms for the scaling.

\begin{proof}[Outline of the proof of Theorem \ref{thm:main}]
 We start by demonstrating the energy estimates available for \eqref{eq:BQ_2} in Section \ref{sec:en_est}. This shows that condition \eqref{eq:unif_energy} can be met naturally (see also Remark \ref{rem:disp_param_indep}).
 
 In Section \ref{sec:disp_eff} we study the dispersive effects in the Boussinesq system \eqref{eq:BQ_2} (Lemma \ref{lem:eigen}, Corollary \ref{cor:eigen} and Proposition \ref{prop:disp}). This is a question regarding the linear part of the equation and inspires a new choice of variables, which diagonalize the linear evolution. In view of this analysis, for more clarity we then reformulate the equations in Section \ref{sec:reform} -- see Corollary \ref{cor:newform}. This provides the foundation for studying the convergence as $\sigma\to\infty$, which is carried out in Section \ref{sec:disp_lim} and results in Proposition \ref{prop:u_conv_abs} and its Corollary \ref{cor:vort_conv_abs}. 
 
 Finally, the only task remaining is to determine the dynamics of the limiting system. This is the subject of Section \ref{sec:self_int}.
\end{proof}

\subsection{Notation}\label{sec:not}
We denote the components of a vector in $\R^3$ either by $x,y,z$ or number indices. The first two components of a vector in $\R^3$ shall be called ``horizontal'' and we introduce the notation of a subscript $h$ to denote the associated quantities and operators (derived from their three-dimensional counterparts). For example, we write $\xi_h=(\xi_1,\xi_2)\in\R^2$, so that $\xi=(\xi_h,\xi_3)\in\R^3$, $\nabla_h=(\partial_{x_1},\partial_{x_2})$ and $\Delta_h:=\partial_{x_1}^2+\partial_{x_2}^2$.

We will use the shorthand $\pm$ in indices as a replacement for either $+$ or $-$, which are then assumed to be used consistently throughout expressions in which they appear.

\section{Energy Estimates}\label{sec:en_est}
As hinted at earlier, it turns out that -- from a perspective of energy estimates -- the natural variables for the Boussinesq system \eqref{eq:BQ_gen} are $u$ and $\frac{\kappa}{\lambda}\theta$: to obtain energy estimates we multiply the equation for $\theta$ by $\frac{\kappa}{\lambda}$ and then test the first equation with $u$ and the second with $\frac{\kappa}{\lambda}\theta$. This yields
\begin{equation*}
\begin{cases}
 &\ip{\partial_t u,u}+\ip{u\cdot\nabla u,u}+\ip{\nabla p,u}=\nu \ip{\Delta u,u}+\kappa^2 \theta u_3,\\
 &\left(\partial_t \frac{\kappa}{\lambda}\theta\right)\left(\frac{\kappa}{\lambda}\theta\right) + \left(u\cdot\nabla \frac{\kappa}{\lambda}\theta\right)\left(\frac{\kappa}{\lambda}\theta\right)=\mu\left(\Delta \frac{\kappa}{\lambda}\theta\right)\left(\frac{\kappa}{\lambda}\theta\right)-\kappa^2 u_3 \theta.
\end{cases}
\end{equation*}
Upon integrating this over $\R^3$ and recalling that $\div \,u=0$ we get
\begin{equation*}
 \begin{cases}
  &\frac{1}{2}\partial_t\abs{u(t)}_{L^2}^2=-\nu\abs{\nabla u(t)}_{L^2}^2+\kappa^2\int_{\R^3}\theta u_3,\\
  &\frac{1}{2}\partial_t\abs{\frac{\kappa}{\lambda}\theta(t)}_{L^2}^2=-\mu\abs{\nabla \frac{\kappa}{\lambda}\theta(t)}_{L^2}^2-\kappa^2\int_{\R^3}\theta u_3.
 \end{cases}
\end{equation*}
Adding these two gives the following energy equality for the perturbed Boussinesq system \eqref{eq:BQ_gen}:
\begin{equation}\label{eq:en_eq}
 \partial_t\abs{u(t)}_{L^2}^2+\partial_t\abs{\frac{\kappa}{\lambda}\theta(t)}_{L^2}^2=-2\nu\abs{\nabla u(t)}_{L^2}^2-2\mu\abs{\nabla \frac{\kappa}{\lambda}\theta(t)}_{L^2}^2.
\end{equation}

As in the introduction we denote by $T:=\frac{\kappa}{\lambda}\theta$ the rescaled version of $\theta$. In terms of the dispersion parameter $\sigma:=\kappa\lambda$ we have $T=\sigma\lambda^{-2}\theta$ and the equations for $(u, T)$ are
\begin{equation}\label{eq:BQ_3}
\begin{cases}
 &\partial_t u+u\cdot\nabla u+\nabla p=\nu\Delta u+\sigma T\vec{e}_3,\\
 &\partial_t T + u\cdot\nabla T=\mu\Delta T-\sigma u_3,\\
 &\div \,u=0,
\end{cases}
\end{equation}
which equals \eqref{eq:BQ_1} in the inviscid case $\mu=\nu=0$.

If in addition we differentiate the above equations and use Gagliardo-Nirenberg interpolation we obtain inequalities for higher order derivative norms:
\begin{lemma}[Energy inequality]
 For any $k\in\N$ there exists a constant $C_k>0$ such that if $(u,T)$ solve \eqref{eq:BQ_3} we have the bound
 \begin{equation}
 \begin{aligned}
  \partial_t\left(\abs{u(t)}_{H^k}^2+\abs{T(t)}_{H^k}^2\right)\leq&-2\nu\abs{\nabla u(t)}_{H^k}^2-2\mu\abs{\nabla T(t)}_{H^k}^2\\
  & +C_k\left(\abs{\nabla u(t)}_{L^\infty}+\abs{\nabla T(t)}_{L^\infty}\right)\left(\abs{u(t)}_{H^k}^2+\abs{T(t)}_{H^k}^2\right).
  \end{aligned}
 \end{equation}
\end{lemma}

\begin{proof}
 This is a standard argument, so we only give a quick sketch of the proof.
 
 For $0\leq l\leq k$, we take a derivative $D^l$ of order $l$ of the equations, multiply by $D^l u$ and $D^l T$ in the respective equations and integrate over $\R^3$. This gives the time derivative of the $L^2$ norms of $D^lu$ and $D^lT$, and their gradients in $L^2$. Upon summation the remaining linear terms vanish and we are only left with the nonlinear pieces $\ip{D^l(u\cdot\nabla u),D^l u}_{L^2}$ and $\ip{D^l(u\cdot\nabla T),D^l T}_{L^2}$. Since $\div\, u=0$, this vanishes if all $l$ derivatives fall onto the gradient term, so by interpolation we can bound these by $\left(\abs{\nabla u}_{L^\infty}+\abs{\nabla T}_{L^\infty}\right)\left(\abs{u}^2_{\dot{H}^l}+\abs{T}^2_{\dot{H}^l}\right)$ (for more details see \cite[Section 4]{SQGpaper}, for example). Now we need only sum over all such derivatives of order $l$ and orders $l\leq k$.
\end{proof}

\begin{remark}[Dependence on $\sigma$]\label{rem:disp_param_indep}
 We note that the energy estimates are uniform in the dispersion parameter $\sigma$, which is natural for a dispersive effect (given by a skew-symmetric singular perturbation) measured in $L^2$ based Sobolev spaces. Most importantly, it implies that condition \eqref{eq:unif_energy} is met naturally in the standard local well-posedness theory.
\end{remark}

In the inviscid case, $\nu=\mu=0$, we can deduce from Gr\"onwall's inequality the growth bound
\begin{equation}
 \abs{u(t)}_{H^k}+\abs{T(t)}_{H^k}\leq \left(\abs{u(0)}_{H^k}+\abs{T(0)}_{H^k}\right)\exp\left(C_k\int_0^t \abs{\nabla u(s)}_{L^\infty}+\abs{\nabla T(s)}_{L^\infty}\;ds\right).
\end{equation}

From this it follows directly that the inviscid Boussinesq system \eqref{eq:BQ_1} is locally well-posed in $H^k$ for $k\geq 3$ (with $H^k$ estimates uniform in $\sigma$).

\begin{remark}[Physical relevance]
In terms of the unrescaled variables in system \eqref{eq:BQ_gen} one may wish to separately take note of two cases for the limit $\sigma=\kappa\lambda\to\infty$, which both guarantee uniform (in $\sigma$) energy estimates as required in \eqref{eq:unif_energy}:
\begin{enumerate}
 \item Fix $\lambda>0$, let $\kappa\to\infty$: Any initial data $(u_0,\theta_0)$ work (as long as they are of small size, as in Theorem \ref{thm:main}). \\This is the case where we fix a stratification profile $u=0$, $\theta=\lambda^2 z$ and let the gravitational force (through the constant $\kappa$) tend to infinity.
 \item Let $\lambda\to\infty$, fix $\kappa>0$: We need $\theta_0=0$. \\Here we fix the strength of gravity $\kappa$ and consider perturbations of increasingly steep stratified solutions $u=0$, $\theta=\lambda^2 z$. Due to the scaling of the energy we have to require that the initial data in $\theta$ vanish, else the energy will not remain uniformly bounded as $\lambda\to\infty$.
\end{enumerate}
\end{remark}

\section{Dispersive Effects in the Inviscid System}\label{sec:disp_eff}
To understand the dispersive effects present in \eqref{eq:BQ_gen} we continue our study of the inviscid case, where $\nu=\mu=0$. For simplicity of notation we henceforth assume the subscripts for $T$ as in the previous section to be understood and thus drop them from the notation. For convenience we recall here the relevant system \eqref{eq:BQ_1},
\begin{equation}\label{eq:BQ_new}
\begin{cases}
 &\partial_t u+u\cdot\nabla u+\nabla p=\sigma T\vec{e}_3,\\
 &\partial_t T + u\cdot\nabla T=-\sigma u_3,\\
 &\div \,u=0,
\end{cases}
\end{equation}
where $\sigma:=\kappa\lambda$ and the pressure $p$ can be expressed in terms of $u$ through the third equation as $p=(-\Delta)^{-1}(\div(u\cdot\nabla u)-\sigma\partial_3 T)$. In view of a decomposition of this system into a linear and a nonlinear part we split the pressure as
\begin{equation*}
 p=p^L+p^{NL},\quad p^L:=-\sigma(-\Delta)^{-1}\partial_3 T,\quad p^{NL}:=(-\Delta)^{-1}\div(u\cdot\nabla u).
\end{equation*}

The linear part of \eqref{eq:BQ_new} then reads
\begin{equation}\label{eq:BQ_new_lin}
 \begin{cases}
 &\partial_t u+\nabla p^L=\sigma T\vec{e}_3,\\
 &\partial_t T =-\sigma u_3,\\
 &\div \,u=0,
\end{cases}
\end{equation}
and we can formally rewrite \eqref{eq:BQ_new} as
\begin{equation}\label{eq:BQ_new2}
\begin{cases}
 &\partial_t u+N^u(u,u)+\nabla p^L=\sigma T\vec{e}_3,\\
 &\partial_t T + M^u(u,T)=-\sigma u_3,\\
 &\div \,u=0,
\end{cases}
\end{equation}
with nonlinearities $N^u(u,u)=u\cdot\nabla u+\nabla p^{NL}$ and $M^u(u,T)=u\cdot\nabla T$.

In the following it will be convenient to also work with the vorticity variable $\omega:=\curl\, u$, since this avoids having to solve for the pressure. The velocity $u$ can then be expressed in terms of the vorticity $\omega$ as $u=(-\Delta)^{-1}\curl\, \omega$. The full, nonlinear Boussinesq system in vorticity formulation is then given by the equations
\begin{equation}\label{eq:Bouss_vort}
\begin{cases}
 &\partial_t \omega+N^\omega(\omega,\omega)=\sigma\left(\begin{array}{c}\partial_y T\\ -\partial_x T\\0\end{array}\right),\\
 &\partial_t T +M^\omega(\omega,T)=\sigma (-\Delta)^{-1}\left(\partial_y\omega_1-\partial_x\omega_2\right),\\
 &\div\,\omega=0,
\end{cases}
\end{equation}
where
\begin{equation*}
N^\omega(\omega,\omega):=\curl \left(u\cdot \nabla u\right)=\omega\cdot\nabla u- u\cdot\nabla{\omega}
\end{equation*}
denotes the Euler nonlinearity in vorticity form and the transport term $u\cdot\nabla T$ has been written as $M^\omega(\omega,T)$. We notice that the first equation implies $\partial_t\div\, \omega=0$, so that the equation $\div\,\omega=0$ reduces to a condition on the initial data.

The associated linear system reads
\begin{equation}\label{eq:vortlin}
\begin{cases}
 &\partial_t \omega=\sigma\left(\begin{array}{c}\partial_y T\\ -\partial_x T\\0\end{array}\right),\\
 &\partial_t T =\sigma (-\Delta)^{-1}\left(\partial_y\omega_1-\partial_x\omega_2\right),\\
 &\div\,\omega_0=0,
\end{cases}
\end{equation}
since also here $\partial_t\div\,\omega=0$ follows from the first equation. 

 The space of solutions to the linear systems \eqref{eq:BQ_new_lin} and \eqref{eq:vortlin} is three-dimensional and can be computed explicitly as follows\footnote{We recall the notation $\pm$ in indices as a replacement for either $+$ or $-$, which are then assumed to be used consistently.}:
\begin{lemma}\label{lem:eigen}
 In Fourier space the linear system \eqref{eq:BQ_new_lin} can be diagonalized for any $\xi\in\R^3$ with $(\xi_1,\xi_2)\neq(0,0)$ in the following three eigenvectors and eigenvalues:
 \begin{equation}
  \begin{aligned}
   &\left(\begin{array}{c}-\xi_2\\ \xi_1\\ 0\\ 0\end{array}\right)\text{ with eigenvalue } 0,\\
   &\left(\begin{array}{c} \xi_1\xi_3\\ \xi_2\xi_3\\ -\abs{\xi_h}^2 \\ -i\abs{\xi_h}\abs{\xi}\end{array}\right)\text{ with eigenvalue }i\frac{\abs{\xi_h}}{\abs{\xi}},\\
   &\left(\begin{array}{c} \xi_1\xi_3\\ \xi_2\xi_3\\ -\abs{\xi_h}^2 \\ i\abs{\xi_h}\abs{\xi}\end{array}\right)\text{ with eigenvalue }-i\frac{\abs{\xi_h}}{\abs{\xi}}.
  \end{aligned}
 \end{equation}
 We denote the corresponding eigenspaces in physical space by $E^u_0$, $E^u_-$ and $E^u_+$, and the projections onto them by $\PP^u_0$, $\PP^u_-$ and $\PP^u_+$, respectively. 

 The corresponding eigenvectors and eigenvalues for the linear system \eqref{eq:vortlin} are
 \begin{equation}
  \begin{aligned}
   &\left(\begin{array}{c}-\xi_1\xi_3\\ -\xi_2\xi_3\\ \abs{\xi_h}^2\\ 0\end{array}\right)\text{ with eigenvalue } 0,\\
   &\left(\begin{array}{c}-\xi_2\abs{\xi}\\ \xi_1\abs{\xi}\\ 0\\ \abs{\xi_h}\end{array}\right)\text{ with eigenvalue } i\frac{\abs{\xi_h}}{\abs{\xi}},\\
   &\left(\begin{array}{c}-\xi_2\abs{\xi}\\ \xi_1\abs{\xi}\\ 0\\ -\abs{\xi_h}\end{array}\right)\text{ with eigenvalue }-i\frac{\abs{\xi_h}}{\abs{\xi}}.
  \end{aligned}
 \end{equation}
 We denote the corresponding eigenspaces in physical space by $E^\omega_0$, $E^\omega_-$ and $E^\omega_+$, and the projections onto them by $\PP^\omega_0$, $\PP^\omega_-$ and $\PP^\omega_+$, respectively. 
\end{lemma}

\begin{remark}
 As will be shown later (see Proposition \ref{prop:disp}), the modes with eigenvalues $\pm i\frac{\abs{\xi_h}}{\abs{\xi}}$ exhibit dispersive decay.
\end{remark}

\begin{proof}
We only give the computation for \eqref{eq:vortlin}, since \eqref{eq:BQ_new_lin} works analogously. 

After taking Fourier transforms, \eqref{eq:vortlin} reads
\begin{equation}
 \partial_t\left(\begin{array}{c}\hat{\omega}_1\\ \hat{\omega}_2\\ \hat{\omega}_3\\ \hat{T}\end{array}\right)=\sigma\underbrace{\left(
\begin{array}{cccc}
0 & 0 & 0 & -i\xi_2\\
0 & 0 & 0 & i\xi_1\\
0 & 0 & 0 & 0\\
-i\frac{\xi_2}{\abs{\xi}^2} & i\frac{\xi_1}{\abs{\xi}^2} & 0  & 0
\end{array}\right)}_{=:A(\xi)}\left(\begin{array}{c}\hat{\omega}_1\\ \hat{\omega}_2\\ \hat{\omega}_3\\ \hat{T}\end{array}\right).
\end{equation}

The matrix $A(\xi)$ is diagonalizable as long as $\xi_1\neq 0$ or $\xi_2\neq 0$. Its eigenvalues and eigenspaces can be computed to be $v_1:=\left(0,0,1,0\right)^\intercal$, $v_2:=\left(\xi_1,\xi_2,1,0\right)^\intercal$ with eigenvalue $0$, $v_3:=\left(-\xi_2\abs{\xi},\xi_1\abs{\xi},0,\abs{\xi_h}\right)^\intercal$ with eigenvalue $i\frac{\abs{\xi_h}}{\abs{\xi}}$ and $v_4:=\left(-\xi_2\abs{\xi},\xi_1\abs{\xi},0,-\abs{\xi_h}\right)^\intercal$ with eigenvalue $-i\frac{\abs{\xi_h}}{\abs{\xi}}$.

A priori the linear system thus has two stationary modes ($v_1$ and $v_2$) and two modes whose time evolution is given by $e^{\pm i\frac{\abs{\xi_h}}{\abs{\xi}}}$ ($v_3$ and $v_4$). We notice that through the condition that $\div\,\omega=0$ (which is automatically propagated) we can reduce the system to one with three degrees of freedom: it is automatically satisfied by $v_3$ and $v_4$, but not individually by $v_1$ and $v_2$. Hence if $s:=\alpha v_1+\beta v_2$ is an element of the eigenspace of the eigenvalue 0 we impose the condition that $\div\,\omega=0$ through
\begin{equation*}
 0=\xi\cdot s=\xi\cdot\left(\alpha v_1+\beta v_2\right)=\xi_3\alpha+(\xi_1^2+\xi_2^2)\beta,
\end{equation*}
so that $\beta=-\frac{\xi_3}{\xi_1^2+\xi_2^2}\alpha$ and thus $s=\alpha\left(-\frac{\xi_1\xi_3}{\xi_1^2+\xi_2^2},-\frac{\xi_2\xi_3}{\xi_1^2+\xi_2^2},1,0\right)^\intercal$ spans the eigenspace of eigenvalue 0.
\end{proof}

By projecting onto the eigenspaces of the linear system one can then directly deduce the following
\begin{cor}\label{cor:eigen}
 A solution $(u,T)$ of \eqref{eq:BQ_new2} or $(\omega,T)$ of \eqref{eq:Bouss_vort} can be decomposed as $$(u,T)(t)=(S(t),0)+D_-(t)+D_+(t) \quad\text{or}\quad (\omega,T)(t)=(s(t),0)+d_-(t)+d_+(t),$$ where $(S(t),0):=\PP^u_0(u,T)(t)\in E^u_0$, $(s(t),0):=\PP^\omega_0(\omega,T)(t)\in E^\omega_0$, $D_\pm(t):=\PP^u_\pm(u,T)(t)\in E^u_\pm$ and $d_\pm(t):=\PP^\omega_\pm(\omega,T)(t)\in E^\omega_\pm$.
 
 More explicitly, there exist functions $\psi(t), a(t)$ and $b(t)$ such that
 \begin{equation}\label{eq:vort_eigen}
 \begin{aligned}
  (S(t),0)&=\left(\begin{array}{c}-\partial_2 \psi\\ \partial_1 \psi\\ 0\\0\end{array}\right), (s(t),0)=\left(\begin{array}{c}-\partial_1\partial_3 \psi\\ -\partial_2\partial_3 \psi\\ \Delta_h \psi\\0\end{array}\right),\\
  D_+(t)&=\left(\begin{array}{c}\partial_1\partial_3 a\\ \partial_2\partial_3 a\\ -\Delta_h a\\ i\abs{\nabla_h}\abs{\nabla}a\end{array}\right), d_+(t)=\left(\begin{array}{c}-\partial_2 \abs{\nabla}a\\ \partial_1 \abs{\nabla}a\\ 0\\ i\abs{\nabla_h}a\end{array}\right),\\
  D_-(t)&=\left(\begin{array}{c}\partial_1\partial_3 b\\ \partial_2\partial_3 b\\ -\Delta_h b\\ -i\abs{\nabla_h}\abs{\nabla}b\end{array}\right), d_-(t)=\left(\begin{array}{c}-\partial_2 \abs{\nabla}b\\ \partial_1 \abs{\nabla}b\\ 0\\ -i\abs{\nabla_h}b\end{array}\right).\\
 \end{aligned}
 \end{equation}
\end{cor}

\begin{remark}
 Here the functions $a$ and $b$ are complex-valued and arise from the projections onto the eigenspaces of the propagators $e^{it\frac{\abs{\nabla_h}}{\abs{\nabla}}}$ and $e^{-it\frac{\abs{\nabla_h}}{\abs{\nabla}}}$, respectively. This is natural since the splitting stems from a diagonalization of a real-valued system with (conjugate) complex eigenvalues on the Fourier side -- compare also the linear wave equation and the half-wave operators that arise when diagonalizing it as a first-order system.\footnote{Alternatively we may rewrite the linear modes using the propagators $\sin\left(t\frac{\abs{\nabla_h}}{\abs{\nabla}}\right)$ and $\cos\left(t\frac{\abs{\nabla_h}}{\abs{\nabla}}\right)$: We know that a solution $(\omega_L,T_L)(t)$ to the linear problem \eqref{eq:vortlin} in the eigenspaces $E_\pm^\omega$ can be written as
 \begin{equation*}
  (\omega_L,T_L)(t)=e^{it\frac{\abs{\nabla_h}}{\abs{\nabla}}}\left(\begin{array}{c}-\partial_2 \abs{\nabla}a_0\\ \partial_1 \abs{\nabla}a_0\\ 0\\ i\abs{\nabla_h}a_0\end{array}\right) + e^{-it\frac{\abs{\nabla_h}}{\abs{\nabla}}}\left(\begin{array}{c}-\partial_2 \abs{\nabla}b_0\\ \partial_1 \abs{\nabla}b_0\\ 0\\ -i\abs{\nabla_h}b_0\end{array}\right),
 \end{equation*}
 with $a_0, b_0:\R^3\to\C$ complex-valued projections of the initial data onto the eigenspaces $E_\pm^\omega$. By expanding and regrouping these we deduce that there exist \emph{real-valued} functions $\alpha_0,\beta_0:\R^3\to\R$ such that
 \begin{equation*}
  (\omega_L,T_L)(t)=\cos\left(t\frac{\abs{\nabla_h}}{\abs{\nabla}}\right)\left(\begin{array}{c}-\partial_2 \abs{\nabla}\alpha_0\\ \partial_1 \abs{\nabla}\alpha_0\\ 0\\ \abs{\nabla_h}\beta_0\end{array}\right)-\sin\left(t\frac{\abs{\nabla_h}}{\abs{\nabla}}\right)\left(\begin{array}{c}-\partial_2 \abs{\nabla}\beta_0\\ \partial_1 \abs{\nabla}\beta_0\\ 0\\ \abs{\nabla_h}\alpha_0\end{array}\right).
 \end{equation*}}
\end{remark} 
 
\begin{remark}[Stream Function] 
 The function $\psi$ of the stationary mode is real-valued and can be identified as the \emph{stream function} of a two-dimensional, incompressible flow (for each fixed $x_3\in\R$).
\end{remark}

\subsection{Dispersive Decay}
We study now the decay properties of the semigroup generated by the linear operators that arise in \eqref{eq:BQ_new_lin} and \eqref{eq:vortlin}. We denote by $L_\pm$ the linear operators with Fourier symbols $\pm i\frac{\abs{\xi_h}}{\abs{\xi}}$, i.e.\ for any (possibly vector valued) Schwartz function $f\in\mathcal{S}(\R^3)$ we let
\begin{equation*}
 L_\pm f(x):=\iFT{\pm i\frac{\abs{\xi_h}}{\abs{\xi}}\hat{f}(\xi)}.
\end{equation*}
For the semigroup generated by $L_\pm$ we then have the following decay estimate:
\begin{prop}\label{prop:disp}
 There exists a constant $C>0$ such that for any $f\in C^\infty_c(\R^3)$
 \begin{equation}\label{eq:disp}
  \abs{e^{L_\pm t}f}_{L^\infty}\leq C t^{-\frac{1}{2}}\abs{f}_{\dot{B}^3_{1,1}}. 
 \end{equation}
\end{prop}
\begin{proof}
 By scaling it suffices to prove the estimate $\abs{e^{L_\pm t}\varphi}_{L^\infty}\lesssim t^{-\frac{1}{2}}$ for a smooth bump function $\varphi\in C^\infty_c(\R^3)$ with Fourier transform supported in an annulus around $\abs{\xi}\sim 1$ (away from the origin -- see \cite[Proof of Proposition 2.1]{SQGpaper} for more details on this reduction). Hence we bound the integral
 \begin{equation}\label{eq:disp_int}
  \int_{\R^3}e^{ix\cdot\xi\pm it\frac{\abs{\xi_h}}{\abs{\xi}}}\varphi(\xi)\,d\xi.
 \end{equation}
 
 With stationary phase techniques in mind we smoothly split the area of integration into two pieces, according to the size of ${\xi_3}$.  To this end let $\epsilon>0$, to be chosen later. 
 
 For $\abs{\xi_3}\leq\epsilon$ we note that
 \begin{equation*}
  \int_{\abs{\xi_3}\leq\epsilon}e^{ix\cdot\xi\pm it\frac{\abs{\xi_h}}{\abs{\xi}}}\varphi(\xi)\,d\xi\lesssim \abs{\varphi}_{L^\infty} \epsilon,
 \end{equation*}
 since $\abs{\{\abs{\xi_3}<\epsilon\}}\lesssim\epsilon$.
 
  If $\abs{\xi_3}>\epsilon$ we apply the method of stationary phase. We note that 
  $$\nabla\frac{\abs{\xi_h}}{\abs{\xi}}=\frac{\xi_3}{\abs{\xi}^3}\left(\frac{\xi_1\xi_3}{\abs{\xi_h}},\frac{\xi_2\xi_3}{\abs{\xi_h}},-\abs{\xi_h}\right),$$
  so for any $(t,x)\in\R\times\R^3$ there at most finitely many stationary points of the exponent in \eqref{eq:disp_int}. Furthermore we compute that $\abs{\det Hess\frac{\abs{\xi_h}}{\abs{\xi}}}^{-\frac{1}{2}}= \abs{\xi_3}^{-2}\abs{\xi_h}^{\frac{1}{2}}\abs{\xi}^{\frac{9}{2}}$, so the standard stationary phase lemma gives a bound of
 \begin{equation*}
  \int_{\abs{\xi_3}>\epsilon}e^{ix\cdot\xi\pm it\frac{\abs{\xi_h}}{\abs{\xi}}}\varphi(\xi)\,d\xi\lesssim \abs{\varphi}_{L^\infty} t^{-\frac{3}{2}}\epsilon^{-2}.
 \end{equation*}
 
 In total we thus have the bound $\epsilon+t^{-\frac{3}{2}}\epsilon^{-2}$, so that the choice of $\epsilon=t^{-\frac{1}{2}}$ gives the claim.
\end{proof}

\begin{remark}\label{rem:disp_param}
 The systems \eqref{eq:BQ_new_lin} and \eqref{eq:vortlin} include the additional parameter $\sigma>0$, which governs the strength of dispersion: if $f(t)$ solves $\partial_t f=\sigma L_{\pm}f$, then $\abs{f(t)}_{L^\infty}\leq C (\sigma t)^{-\frac{1}{2}}\abs{f}_{\dot{B}^3_{1,1}}$.
\end{remark}

\section{A Reformulation of the Equations}\label{sec:reform}
In this section we rewrite the Boussinesq system \eqref{eq:Bouss_vort} using the above analysis to set up the problem of studying the behavior for strong dispersion

Now we project the full, nonlinear Boussinesq system onto the eigenmodes of the corresponding linear system. The steps and notation in both the velocity and vorticity formulations (\eqref{eq:BQ_new_lin} and \eqref{eq:BQ_new2}, or \eqref{eq:vortlin} and \eqref{eq:Bouss_vort}, respectively) are completely analogous -- we give the relevant details here in the vorticity formulation, since we will use it later to compute the limiting dynamics in Section \ref{sec:self_int}. The results are summarized in Corollary \ref{cor:newform}.

\begin{notation}
We begin by introducing the ``bar'' notation to denote the first three components of the projections onto the eigenmodes of the linear systems \eqref{eq:BQ_new_lin} and \eqref{eq:vortlin}, so that we have $\bar{\PP}^u_0, \bar{\PP}^u_\pm, \bar{\PP}^\omega_0, \bar{\PP}^\omega_\pm:\R^4\to\R^3$. In particular then we have $S(t)=\bar{\PP}^u_0(u,T)(t)$ and $s(t)=\bar{\PP}^\omega_0(\omega,T)(t)$, since the $T$ component of the stationary mode vanishes (as we saw in Lemma \ref{lem:eigen}).
\end{notation}

\textbf{Stationary Modes.} We apply the projection $\PP^\omega_0$ to \eqref{eq:Bouss_vort}. We recall that the stationary mode $\PP^\omega_0\omega$ does not involve $T$, so the only nonlinearity present there is 
$$P^\omega(\omega,\omega):=\bar{\PP}^\omega_0(N^\omega(\omega,\omega))=\bar{\PP}^\omega_0\left(N^\omega(\bar{\PP}^\omega_0\omega+\bar{\PP}^\omega_+\omega+\bar{\PP}^\omega_-\omega,\bar{\PP}^\omega_0\omega+\bar{\PP}^\omega_+\omega+\bar{\PP}^\omega_-\omega)\right),$$
which (by bilinearity of $N$) can be split into nine pieces, in accordance with the types of interactions of stationary and dispersive modes. To simplify notation we use the shorthands $\ast$, $+$ or $-$ to denote an input of the stationary or dispersive modes, respectively. Hence $P^\omega(\ast,\ast)=\bar{\PP}^\omega_0\left(N^\omega(\bar{\PP}^\omega_0\omega,\bar{\PP}^\omega_0\omega)\right)$, $P^\omega(\ast,+)=\bar{\PP}^\omega_0\left(N^\omega(\bar{\PP}^\omega_0\omega,\bar{\PP}^\omega_+\omega)\right)$ etc. One can then write $P^\omega(\omega,\omega)=\sum_{j,k=\ast,+,-}P^\omega(j,k)$.

\textbf{Dispersive Modes.} Similarly we split the nonlinearities projected onto the dispersive modes as
\begin{equation}
\begin{aligned}
 &Q^\omega_+\left((\omega,T),(\omega,T)\right):=\PP^\omega_+\left((N^\omega(\omega,\omega),M^\omega(\omega,T))\right),\\
 &Q^\omega_-\left((\omega,T),(\omega,T)\right):=\PP^\omega_-\left((N^\omega(\omega,\omega),M^\omega(\omega,T))\right),
\end{aligned}
\end{equation}
and write $Q^\omega_+(\ast,\pm)=Q^\omega_+\left(\PP^\omega_0(\omega,T),\PP^\omega_\pm(\omega,T)\right)$ etc.

\begin{cor}\label{cor:newform}
 The Boussinesq system in velocity form \eqref{eq:BQ_new} can be rewritten as
 \begin{equation}\label{eq:BQ_abs}
 \begin{cases}
 \partial_t S(t)&=\sum_{j,k=\ast,+,-}P^u(j,k),\\
 \partial_t D_\pm(t)+\sigma L_\pm D_\pm(t)&=\sum_{j,k=\ast,+,-}Q^u_\pm(j,k),
\end{cases}
\end{equation}
with initial data $(u,T)(0)=(u_0,T_0)$ satisfying $\div \,u_0=0$, and in vorticity form \eqref{eq:Bouss_vort} as
\begin{equation}\label{eq:Bouss_new}
 \begin{cases}
 \partial_t s(t)&=\sum_{j,k=\ast,+,-}P^\omega(j,k),\\
 \partial_t d_\pm(t)+\sigma L_\pm d_\pm(t)&=\sum_{j,k=\ast,+,-}Q^\omega_\pm(j,k),
\end{cases}
\end{equation}
with the initial data $(\omega,T)(0)=(\omega_0,T_0)$ satisfying $\div \,\omega_0=0$. (We recall that the second equations in both these systems stand for two separate equations with either the choice of $+$ or $-$ indices.)
\end{cor}

\begin{remark}
 Note that in this formulation the conditions that $u$ and $\omega$ be divergence free are automatically built in, since we project onto divergence-free vector fields, by Lemma \ref{lem:eigen}.
\end{remark}

\section{Limit of Strong Dispersion}\label{sec:disp_lim}
This section describes the effect of strong dispersion in \eqref{eq:BQ_new}. More precisely, we work on a \emph{fixed time interval}\footnote{Without loss of generality we assume this to be $[0,1]$.} for which one has \emph{uniform} energy estimates in the dispersion parameter $\sigma>0$, as is the setting of Theorem \ref{thm:main}, and study the dynamics as $\sigma\to\infty$. Intuitively one expects the system to decouple into a purely dispersive part (that tends to zero in a dispersive norm) and a limiting ``stationary'' part. Later on, in Section \ref{sec:self_int} we will identify this limiting system as stratified 2d Euler equations.

Using the analysis from Section \ref{sec:disp_eff} we can prove:
\begin{prop}\label{prop:u_conv_abs}
 Consider the equations \eqref{eq:BQ_abs} for the Boussinesq system,
 \begin{equation}\label{eq:BQ_abs2}
 \begin{cases}
 \partial_t S^\sigma(t)&=\sum_{j,k=\ast,+,-}P^u(j,k),\\
 \partial_t D^\sigma_\pm(t)+\sigma L_\pm D^\sigma_\pm(t)&=\sum_{j,k=\ast,+,-}Q^u_\pm(j,k),\\
 (u^\sigma,T^\sigma)(0)&=(u_0,T_0),\\
 \div \,u_0&=0,
\end{cases}
\end{equation}
and assume that for some $M>0$ we have the estimates
\begin{equation*}
 \abs{S^\sigma(t)}_{H^{6}},\abs{D^\sigma_\pm(t)}_{H^{6}}\leq M \text{ for } t\in [0,1],
\end{equation*}
as well as $\abs{D^\sigma_\pm(0)}_{\dot{B}^{3}_{1,1}},\abs{D^\sigma_\pm(0)}_{\dot{B}^{4}_{1,1}}\leq M$.

Then as $\sigma\to\infty$, for any $t\in(0,1]$ we have $D^\sigma_\pm(t)\to 0$ in $W^{1,\infty}$ and $S^\sigma(t)\to S^\infty(t)$ in $L^2$, where $S^\infty$ solves
\begin{equation}\label{eq:u_limit}
 \begin{cases}
  \partial_t S^\infty(t)&=P^u(S^\infty,S^\infty),\\
  S^\infty(0)&=\bar{\PP}^u_0(u_0,T_0).
 \end{cases}
\end{equation}
\end{prop}

\begin{remark}
 Note that the evolution of the limiting system \eqref{eq:u_limit} is given by the nonlinear self-interaction of the stationary mode with itself, projected again onto the stationary mode.
\end{remark}

\begin{remark}[Rate of Dispersive Decay]
 The proof below shows that -- as long as there is some decay -- the exact rate of decay in the dispersive estimate \eqref{eq:disp} is not crucial for this result. Similarly, the general idea of this decoupling in the limit of strong dispersion applies in a wider context of interacting oscillatory and dispersive systems.
\end{remark}

The proof combines the dispersive decay with the Duhamel formula.
\begin{proof}
 We first prove the decay of $D^\sigma_\pm$ and then use it to deduce the convergence to the limiting system \eqref{eq:u_limit}. To illustrate that this result does not depend on the precise rate of dispersive decay of the semigroup $e^{\sigma L_\pm t}$ we write $\alpha$ for the exponent $\frac{1}{2}$ in the decay estimate \eqref{eq:disp} and notice that the proof works for any $\alpha>0$.

 \paragraph{Decay of the Dispersive Modes} Applying the dispersive decay estimate \eqref{eq:disp} in Duhamel's formula gives for any $0<t\leq 1$
 \begin{equation}\label{eq:DHdecay}
 \begin{aligned}
  \abs{D^\sigma_\pm(t)}_{L^\infty}&\leq\abs{e^{\sigma tL_\pm}D^\sigma_\pm(0)}_{L^\infty}+\sum_{j,k=\ast,+,-}\int_0^t\abs{e^{\sigma(t-s)L_\pm}Q^u_\pm(j,k)}_{L^\infty}\,ds\\
  &\lesssim (t\sigma)^{-\alpha}\abs{D^\sigma_\pm(0)}_{\dot{B}^3_{1,1}}+\sigma^{-\alpha}\int_0^t (t-s)^{-\alpha} \abs{Q^u_\pm(j,k)}_{\dot{B}^3_{1,1}}\,ds\\
  &\lesssim (t\sigma)^{-\alpha}\abs{D^\sigma_\pm(0)}_{\dot{B}^3_{1,1}}+\sigma^{-\alpha}\int_0^t (t-s)^{-\alpha} \left(\abs{S^\sigma(s)}_{H^5}^2+\abs{D^\sigma_\pm(s)}_{H^5}^2\right)\,ds\\
  &\lesssim (t\sigma)^{-\alpha}\abs{D^\sigma_\pm(0)}_{\dot{B}^3_{1,1}}+\sigma^{-\alpha}M^2\int_0^t(t-s)^{-\alpha}\,ds\\
  &\lesssim \sigma^{-\alpha}\to 0\quad\text{as }\sigma\to\infty.
 \end{aligned}
 \end{equation}
Here we have used the fact\footnote{In fact, for any $\nu>0$ one has the embedding $W^{3+\nu,1}\hookrightarrow\dot{B}^3_{1,1}$.} that $\abs{Q^u_\pm(j,k)}_{\dot{B}^3_{1,1}}\lesssim\abs{Q^u_\pm(j,k)}_{W^{4,1}}\leq\abs{D^\sigma_\pm}_{H^{5}}^2+\abs{S^\sigma}_{H^{5}}^2$ for $j,k=\ast,+,-$.

The same argument works for a derivative of $D^\sigma_\pm$, so that this finishes the proof of the decay in $W^{1,\infty}$ of $D^\sigma_\pm(t)$. 

 \paragraph{Convergence of the Stationary Modes} The initial data of $S^\infty$ and $S^\sigma$ agree and we can write the equation for their difference $S^\sigma-S^\infty$ as
\begin{equation}\label{eq:stat_DH}
 \partial_t\left(S^\sigma(t)-S^\infty(t)\right)=\left(P^u(\ast,\ast)-P^u(S^\infty,S^\infty)\right) + \sum_{j,k=\ast,+,-, \;(j,k)\neq(\ast,\ast)} P^u(j,k).
\end{equation}
Here we have grouped the terms with two stationary inputs first. Hence every term in the second sum on the right contains at least one dispersive term, which we can estimate in $L^\infty$ when estimating the whole expression in $L^2$. More precisely, we note that for such a term with $(j,k)\neq(\ast,\ast)$
$$ \abs{P^u(j,k)}_{L^2}\leq\abs{D^\sigma_\pm(t)}_{W^{1,\infty}}\left(\abs{S^\sigma(t)}_{H^1}+\abs{D^\sigma_\pm(t)}_{H^1}\right)\lesssim M\sigma^{-\alpha}$$
by our previous estimate \eqref{eq:DHdecay}.

As for the first term, the divergence structure of the Euler nonlinearity $u\cdot\nabla u$ allows us to estimate this in $L^2$ without losing a derivative:
\begin{equation}\label{eq:stat_conv}
\begin{aligned}
 &\ip{P^u(\ast,\ast)-P^u(S^\infty,S^\infty),S^\sigma-S^\infty}_{L^2}\\
 &\hspace{2cm}=\ip{\bar{\PP}^u_0 N^u(S^\sigma,S^\sigma)-\bar{\PP}^u_0 N^u(S^\infty,S^\infty),S^\sigma-S^\infty}_{L^2}\\
 &\hspace{2cm} =\ip{\bar{\PP}^u_0\left(N^u(S^\sigma,S^\sigma)-N^u(S^\infty,S^\infty)\right),S^\sigma-S^\infty}_{L^2}\\
 &\hspace{2cm} =\ip{N^u(S^\sigma,S^\sigma)-N^u(S^\infty,S^\infty),S^\sigma-S^\infty}_{L^2}\\
 &\hspace{2cm} \leq\abs{\nabla S^\sigma(t)}_{L^\infty}\abs{S^\sigma(t)-S^\infty(t)}_{L^2}^2,
\end{aligned} 
\end{equation}
where the last inequality follows from the observation that
\begin{equation}\label{eq:cancel1}
\begin{aligned}
 \ip{N^u(S^\sigma,S^\sigma)-N^u(S^\infty,S^\infty),S^\sigma-S^\infty}_{L^2}&=\ip{S^\sigma\cdot\nabla S^\sigma-S^\infty\cdot\nabla S^\infty,S^\sigma-S^\infty}_{L^2} \\
 &\qquad+\ip{\nabla p^{NL}-\nabla p^{NL,\infty},S^\sigma-S^\infty}_{L^2}\\
 &\hspace{-5cm}=\ip{S^\sigma\cdot\nabla S^\sigma-S^\infty\cdot\nabla S^\infty,S^\sigma-S^\infty}_{L^2}\\
 &\hspace{-5cm}=\ip{(S^\sigma-S^\infty)\cdot\nabla S^\sigma+S^\infty\cdot\nabla(S^\sigma-S^\infty),S^\sigma-S^\infty}_{L^2}\\
 &\hspace{-5cm}=\ip{(S^\sigma-S^\infty)\cdot\nabla S^\sigma,S^\sigma-S^\infty}_{L^2}\\
 &\hspace{-5cm}\leq \abs{\nabla S^\sigma}_{L^\infty} \abs{S^\sigma-S^\infty}_{L^2}^2.
\end{aligned}
\end{equation}
This holds since by construction $S^\sigma$ and $S^\infty$ are divergence free, hence the pressure terms vanish and the last equality in \eqref{eq:cancel1} holds.

Hence we deduce using \eqref{eq:stat_DH} that
\begin{equation*}
\begin{aligned}
 &\partial_t\abs{S^\sigma(t)-S^\infty(t)}^2_{L^2}=2\ip{\partial_t(S^\sigma(t)-S^\infty(t)),S^\sigma(t)-S^\infty(t)}\\
 &\lesssim M\sigma^{-\alpha}\abs{S^\sigma(t)-S^\infty(t)}_{L^2}+\abs{\nabla S^\sigma(t)}_{L^\infty} \abs{S^\sigma(t)-S^\infty(t)}^2_{L^2},
\end{aligned}
\end{equation*}
from which Gr\"onwall's Lemma yields
\begin{equation*}
 \abs{S^\sigma(t)-S^\infty(t)}_{L^2}\lesssim tM\sigma^{-\alpha}e^{\int_0^t\abs{\nabla S^\sigma(\tau)}_{L^\infty}\;d\tau}\lesssim tM\sigma^{-\alpha}e^{M t},
\end{equation*}
since $S^\sigma(0)=S^\infty(0)$.
Clearly this tends to zero as $\sigma\to\infty$.
\end{proof}

We can now use this result to show that the corresponding convergence also holds in vorticity, i.e.\ for equation \eqref{eq:Bouss_vort}. This cannot be done purely at the level of the vorticity, but rather builds on the convergence in velocity, since the equivalent of the crucial inequality \eqref{eq:stat_conv} includes terms involving both the velocity and the vorticity.
\begin{cor}\label{cor:vort_conv_abs}
 In the setting of Proposition \ref{prop:u_conv_abs}, consider the equations \eqref{eq:Bouss_new} for the Boussinesq system in vorticity form,
 \begin{equation}\label{eq:Bouss_new2}
 \begin{cases}
 \partial_t s^\sigma(t)&=\sum_{j,k=\ast,+,-}P^\omega(j,k),\\
 \partial_t d^\sigma_\pm(t)+\sigma L_\pm d^\sigma_\pm(t)&=\sum_{j,k=\ast,+,-}Q^\omega_\pm(j,k),\\
 (\omega^\sigma,T^\sigma)(0)&=(\omega_0,T_0),\\
 \div \,\omega_0&=0,
\end{cases}
\end{equation}
and assume now that for some $M>0$
\begin{equation*}
 \abs{s^\sigma(t)}_{H^6},\abs{d^\sigma_\pm(t)}_{H^6}\leq M\text{ for } t\in [0,1],
\end{equation*}
and $\abs{d^\sigma_\pm(0)}_{\dot{B}^{3}_{1,1}},\abs{d^\sigma_\pm(0)}_{\dot{B}^{4}_{1,1}}\leq M$.

Then as $\sigma\to\infty$, for any $t\in(0,1]$ we have $d^\sigma_\pm(t)\to 0$ in $W^{1,\infty}$ and $s^\sigma(t)\to s^\infty(t)$ in $L^2$, where $s^\infty$ solves
\begin{equation}\label{eq:vort_limit}
 \begin{cases}
  \partial_t s^\infty(t)&=P^\omega(s^\infty,s^\infty),\\
  s^\infty(0)&=\bar{\PP}^\omega_0\omega_0.
 \end{cases}
\end{equation}
\end{cor}

\begin{proof}
 Since $\omega=\curl \,u$ one may apply the curl in the corresponding calculations for the proof of Proposition \ref{prop:u_conv_abs} or work directly with the reformulation \eqref{eq:Bouss_new2}. In either case, essentially one ends up having to bound the analogue of \eqref{eq:stat_conv}, for which one calculates that
 \begin{equation}\label{eq:stat_conv_vort}
\begin{aligned}
 &\ip{P^\omega(\ast,\ast)-P^\omega(s^\infty,s^\infty),s^\sigma-s^\infty}_{L^2}=\ip{\bar{\PP}^\omega_0 N^\omega(s^\sigma,s^\sigma)-\bar{\PP}^\omega_0 N^\omega(s^\infty,s^\infty),s^\sigma-s^\infty}_{L^2}\\
 &\quad =\ip{\bar{\PP}^\omega_0\left(N^\omega(s^\sigma,s^\sigma)-N^\omega(s^\infty,s^\infty)\right),s^\sigma-s^\infty}_{L^2}\\
 &\quad =\ip{N^\omega(s^\sigma,s^\sigma)-N^\omega(s^\infty,s^\infty),s^\sigma-s^\infty}_{L^2}\\
 &\quad \leq\abs{\nabla s^\sigma(t)}_{L^\infty}\abs{s^\sigma(t)-s^\infty(t)}_{L^2}^2+\abs{\nabla s^\sigma(t)}_{L^\infty}\abs{S^\sigma(t)-S^\infty(t)}_{L^2}\abs{s^\sigma(t)-s^\infty(t)}_{L^2}.
\end{aligned} 
\end{equation}
This is due to the fact that
\begin{equation}\label{eq:cancel2}
\begin{aligned}
 &\ip{N^\omega(s^\sigma,s^\sigma)-N^\omega(s^\infty,s^\infty),s^\sigma-s^\infty}_{L^2}\\
 &\;=\ip{s^\sigma\cdot\nabla S^\sigma-S^\sigma\cdot\nabla s^\sigma-s^\infty\cdot\nabla S^\infty+S^\infty\cdot\nabla s^\infty,s^\sigma-s^\infty}_{L^2}\\
 &\;=\ip{(s^\sigma-s^\infty)\cdot\nabla S^\sigma,s^\sigma-s^\infty}_{L^2}+\ip{s^\infty\cdot\nabla(S^\sigma-S^\infty),s^\sigma-s^\infty}_{L^2}\\
 &\quad-\ip{S^\sigma\cdot\nabla(s^\sigma-s^\infty),s^\sigma-s^\infty}_{L^2}-\ip{(S^\sigma-S^\infty)\cdot\nabla s^\infty,s^\sigma-s^\infty}_{L^2}\\
 &\;=\ip{(s^\sigma-s^\infty)\cdot\nabla S^\sigma,s^\sigma-s^\infty}_{L^2}+\ip{s^\infty\cdot\nabla(S^\sigma-S^\infty),s^\sigma-s^\infty}_{L^2}\\
 &\quad-\ip{(S^\sigma-S^\infty)\cdot\nabla s^\infty,s^\sigma-s^\infty}_{L^2},\\
\end{aligned}
\end{equation}
and allows us to argue as in the proof of Proposition \ref{prop:u_conv_abs}, once we invoke the $L^2$ convergence of $S^\sigma\to S^\infty$ established therein.
\end{proof}

\section{Self-Interaction of the Stationary Mode}\label{sec:self_int}
As we saw in Section \ref{sec:disp_lim}, as the strength of the dispersion increases, the dynamics of the Boussinesq system decouple into a purely dispersive part (which decays) and a stationary part, governed by limiting systems \eqref{eq:u_limit} or \eqref{eq:vort_limit} in velocity or vorticity form, respectively. As remarked above, these limiting systems consist of the nonlinear interaction of the stationary mode with itself, whilst outputting again the stationary mode. In this section we show that these systems are in fact two-dimensional Euler equations. We present the relevant computations here in the vorticity formulation, since one can then avoid having to solve for the pressure.

For this we consider a velocity $S_\psi$ and its associated vorticity $s_\psi$ of the form $S_\psi=\left(\begin{array}{c}-\partial_2 \psi\\ \partial_1 \psi\\ 0\end{array}\right)$ for some $\psi(t,x)$ and compute the projection of $\curl\left(S_\psi\cdot\nabla S_\psi\right)$ onto $E^\omega_0$, i.e.\ we compute $P^\omega(s_\psi,s_\psi)$ as it appears in \eqref{eq:vort_limit}.
\begin{lemma}[Stationary Output]\label{lem:statout}
 The stationary output from the self-interaction of a stationary mode is  
 \begin{equation}\label{eq:limit_dyn1}
  P^\omega(s_\psi,s_\psi)=s_H
 \end{equation}
 with
 \begin{equation}\label{eq:limit_dyn2}
  H(t)=\Delta_h^{-1}\left[\partial_1\psi(t)\,\Delta_h\partial_2\psi(t)-\partial_2\psi(t)\,\Delta_h\partial_1\psi(t)\right].
 \end{equation}
\end{lemma}

\begin{proof}
 To clarify the notation we recall that $$P^\omega(s_\psi,s_\psi)=\bar{\PP}^\omega_0 N^\omega(s_\psi,s_\psi)=\bar{\PP}^\omega_0\curl\left(S_\psi(t)\cdot \nabla S_\psi(t)\right).$$
 By construction we must have $P^\omega(s_\psi,s_\psi)=s_H$ for some $H$, so the key is to notice that by Corollary \ref{cor:eigen} we can identify $H(t)$ through the requirement 
 \begin{equation*}
  \Delta_h H(t)=\curl\left(S_\psi(t)\cdot \nabla S_\psi(t)\right)_3.
 \end{equation*}
 This we now compute explicitly.
 
 We have 
 \begin{equation*}
  \begin{aligned}
   \left(S_\psi\cdot\nabla S_\psi\right)_1&=\partial_2\psi\partial_1\partial_2\psi-\partial_1\psi\partial_2^2\psi,\\
   \left(S_\psi\cdot\nabla S_\psi\right)_2&=-\partial_2\psi\partial_1^2\psi+\partial_1\psi\partial_1\partial_2\psi,
  \end{aligned}
 \end{equation*}
 and hence
 \begin{equation*}
  \begin{aligned}
   \left[\curl\left(S_\psi(t)\cdot \nabla S_\psi(t)\right)\right]_3&=\partial_1\left(S_\psi(t)\cdot \nabla S_\psi(t)\right)_2-\partial_2\left(S_\psi(t)\cdot \nabla S_\psi(t)\right)_1\\
   &\hspace{-3cm} =-\partial_1\partial_2\psi\partial_1^2\psi-\partial_2\psi\partial_1^3\psi+\partial_1^2\psi\partial_1\partial_2\psi+\partial_1\psi\partial_1^2\partial_2\psi\\
   &\hspace{-3cm}\qquad -\partial_2^2\psi\partial_1\partial_2\psi-\partial_2\psi\partial_1\partial_2^2\psi+\partial_1\partial_2\psi\partial_2^2\psi+\partial_1\psi\partial_2^3\psi\\
   &\hspace{-3cm} =-\partial_1^3\psi\partial_2\psi+\partial_1^2\psi(-\partial_1\partial_2\psi+\partial_1\partial_2\psi)+\partial_1\psi[\partial_1^2\partial_2\psi+\partial_2^3\psi]\\
   &\hspace{-3cm}\qquad +\partial_2^2\psi[\partial_1^2\partial_2\psi+\partial_2^3\psi]+\partial_2^2\psi[\partial_1\partial_2\psi-\partial_1\partial_2\psi]-\partial_1\partial_2^2\psi\partial_2\psi\\
   &\hspace{-3cm} =-\partial_1^3\psi\partial_2\psi+\partial_1\psi[\Delta_h\partial_2\psi]-\partial_1\partial_2^2\psi\partial_2\psi\\
   &\hspace{-3cm} =\partial_1\psi[\Delta_h\partial_2\psi]-\partial_2\psi[\Delta_h\partial_1\psi].
  \end{aligned}
 \end{equation*}
 From the structure of the eigenspaces (as given in Corollary \ref{cor:eigen}) we deduce that this expression must equal $\Delta_h H$, i.e.\ we can express $H$ as
 \begin{equation*}
  H(t)=\Delta_h^{-1}\left[\partial_1\psi\Delta_h\partial_2\psi-\partial_2\psi\Delta_h\partial_1\psi\right],
 \end{equation*}
 as claimed.
\end{proof}

\begin{proof}[End of the Proof of Theorem \ref{thm:main}]
Using \eqref{eq:limit_dyn1} and \eqref{eq:limit_dyn2} we can rewrite the dynamics of the limiting equation \eqref{eq:vort_limit} in vorticity form as 
\begin{equation}
 \partial_t\psi=-\Delta_h^{-1}\left[\partial_1\psi\Delta_h\partial_2\psi-\partial_2\psi\Delta_h\partial_1\psi\right].
\end{equation}
We recognize here the 2d Euler equations in vorticity form: We have 
\begin{equation*}
\begin{aligned}
 \partial_t\Delta_h\psi&=-\left[\partial_1\psi\partial_2\Delta_h\psi-\partial_2\psi\partial_1\Delta_h\psi\right]
 &=-\nabla_h^\perp\psi\cdot\nabla_h\Delta_h\psi.
\end{aligned} 
\end{equation*}
Setting $\bar{\omega}:=\Delta_h\psi$ and $\bar{u}:=\nabla_h^\perp\psi=\left(\begin{array}{c}-\partial_2\psi\\ \partial_1\psi\end{array}\right)$ we can write this as
\begin{equation}
 \partial_t\bar{\omega}+\bar{u}\cdot\nabla_h\bar{\omega}=0,
\end{equation}
which is the vorticity formulation of \eqref{eq:2DEuler} in Theorem \ref{thm:main}. Notice that since $\psi$ depends on all three space variables (and not just the horizontal ones), this is really a stratified system of 2d Euler equations, i.e.\ for every $x_3\in\R$ we have an individual 2d Euler initial value problem in the horizontal variables.

This completes the proof of Theorem \ref{thm:main}.
\end{proof}

\begin{remark}
 For any finite $\sigma>0$, the self-interaction of the stationary mode also produces decaying modes, which can be computed directly by projecting onto the relevant subspaces. Since this is not needed here we leave the calculations to the interested reader.
\end{remark}

\begin{acknowledgements}
 The author would like to thank Tarek Elgindi and Pierre Germain for their helpful comments during many joint discussions.
\end{acknowledgements}

\newpage
\appendix
\bibliographystyle{plain}
\bibliography{3dboussinesq.bib}

\end{document}